\documentclass[11pt,centertags,reqno,twoside]{amsart}
\usepackage{amssymb}
\usepackage{mathptmx}
\usepackage{graphicx}

\usepackage[english]{babel}

\DeclareSymbolFont{SY}{U}{psy}{m}{n}
\DeclareMathSymbol{\emptyset}{\mathord}{SY}{'306}

\newcommand{\dist}{\operatorname{dist}}
\newcommand{\lal}{{\langle}}
\newcommand{\ral}{{\rangle}}
\newcommand{\spec}{\operatorname{spec}}

\newcommand{\bbC}{{\mathbb C}}
\newcommand{\bbR}{{\mathbb R}}

\newcommand{\cB}{{\mathcal B}}
\newcommand{\cD}{{\mathcal D}}
\newcommand{\cG}{{\mathcal G}}

\newcommand{\cO}{{\mathcal O}}

\newcommand{\cV}{{\mathcal V}}

\newcommand{\fb}{\mathfrak{b}}
\newcommand{\fc}{\mathfrak{c}}

\newcommand{\fH}{\mathfrak{H}}
\newcommand{\fh}{\mathfrak{h}}

\newcommand{\fL}{\mathfrak{L}}
\newcommand{\fM}{\mathfrak{M}}

\newcommand{\fs}{\mathfrak{s}}
\newcommand{\fx}{\mathfrak{x}}
\newcommand{\fy}{\mathfrak{y}}

\newcommand{\sE}{\mathsf{E}}

\allowdisplaybreaks

\numberwithin{equation}{section}

\newtheorem{theorem}{Theorem}[section]
\newtheorem{corollary}[theorem]{Corollary}
\newtheorem{lemma}[theorem]{Lemma}

\newtheorem{hypothesis}[theorem]{Hypothesis}
\theoremstyle{definition}
\newtheorem{definition}[theorem]{Definition}
\theoremstyle{remark}
{\it}{\rm}
\newtheorem{remark}[theorem]{Remark}
\newtheorem{example}[theorem]{Example}

\begin{document}

\title[Solvability of the operator Riccati equation in the Feshbach case]
{Solvability of the operator Riccati equation in the Feshbach case}%

\author[S. Albeverio and A. K. Motovilov]
{Sergio Albeverio and  Alexander K. Motovilov}

\address{\hspace*{-1.05cm} Sergio Albeverio \newline
Institut f\"ur Angewandte Mathematik and HCM\newline
Universit\"at Bonn \newline
Endenicher Allee 60, D-53115 Bonn \newline
Germany \newline
\textit{E-mail address: \emph{albeverio@iam.uni-bonn.de}}}

\address{\hspace*{-1.05cm} Alexander K. Motovilov \newline
Bogoliubov Laboratory of Theoretical Physics, JINR  \newline
Joliot-Cu\-rie 6, 141980 Dubna, Moscow Region \newline
Russia\newline
\textit{\emph{and}} \newline
Faculty of Natural and Engineering Sciences \newline
Dubna State University \newline
Universitetskaya 19, 141980 Dubna, Moscow Region \newline
Russia \newline
\textit{E-mail address: \emph{motovilv@theor.jinr.ru}}}

\date{December 14, 2017}


.

\begin{abstract}
We consider a bounded block operator matrix of the form
$$
L=\left(\begin{array}{cc}
A & B \\
C & D
\end{array}
\right),
$$
where the main-diagonal entries $A$ and $D$ are self-adjoint
operators on Hilbert spaces $\fH_{_A}$ and $\fH_{_D}$, respectively;
the coupling $B$ maps $\fH_{_D}$ to $\fH_{_A}$ and $C$ is an
operator from $\fH_{_A}$ to $\fH_{_D}$. It is assumed that the
spectrum $\sigma_{_D}$ of $D$ is absolutely continuous and uniform,
being presented by a single band $[\alpha,\beta]\subset\bbR$,
$\alpha<\beta$, and the spectrum $\sigma_{_A}$ of $A$ is embedded
into $\sigma_{_D}$, that is, $\sigma_{_A}\subset(\alpha,\beta)$. In
its spectral representation, the entry $D$ reads as the operator of
multiplication by the independent variable $\lambda\in(a,b)$. One
more assumption is that both the couplings $B$ and $C$ are defined
via operator-valued functions of $\lambda\in(\alpha,\beta)$ that are
real analytic on $(\alpha,\beta)$ and admit analytic continuation
onto some domain in $\bbC$. This allows one to perform a complex
deformation of $L$. The latter involves, in particular, the
replacement of the original entry $D$ with the operators of
multiplication by the complex variable $\lambda$ running through
piecewise smooth Jordan contours obtained from the interval
$(\alpha,\beta)$ by continuous transformations. We formulate
conditions under which there are bounded solutions to the operator
Riccati equations associated with the complexly deformed block
operator matrix $L$; in such a case the deformed operator matrix $L$
admits a block diagonalization. The same conditions also ensure the
Markus-Matsaev-type factorization of the Schur complement
$M_{_A}(z)=A-z-B(D-z)^{-1}C$\, analytically continued onto the
unphysical sheet(s) of the complex $z$ plane adjacent to the band
$[\alpha,\beta]$. We prove that the operator roots of the continued
Schur complement $M_{_A}$ are explicitly expressed through the
respective solutions to the deformed Riccati equations.

\end{abstract}

\maketitle

\section{Introduction}
\label{SIntro}

Assume that $L$ is a bounded linear operator on a Hilbert space $\fH$.
Suppose $\fH$ is decomposed into the orthogonal sum
\begin{equation}
\label{fH}
{\fH=\fH_{_A}\oplus\fH_{_D}}
\end{equation}
of two subspaces $\fH_{_A}$ and $\fH_{_D}$. Then, with respect to
the decomposition \eqref{fH}, the operator $L$ reads as $2\times2$
block matrix,
\begin{equation} \label{L} L=\left(\begin{array}{cc}
A & B \\
C & D
\end{array}\right),
\end{equation}
where the main-diagonal entries $A$ and $D$ are operators
respectively on $\fH_{_A}$ and $\fH_{_D}$; the coupling $B$ maps
$\fH_{_D}$ to $\fH_{_A}$ and $C$ is an operator from $\fH_A$ to
$\fH_D$. The relations
\begin{align}
\label{RicA}
{XA-DX+XBX}&{=C}, \qquad X:\,\fH_{_A}\to\fH_{_D},\\
\label{RicD}
{YD-AY+YCY}&{=B}, \qquad Y:\,\fH_{_D}\to\fH_{_A},
\end{align}
are called the (pair of dual) operator Riccati equations associated
with the block operator matrix $L$.

It is well known (see, e.g., \cite{AMM}, \cite{MSS}) that a bounded
operator $X$ from $\fH_{_A}$ to $\fH_{_D}$ is a solution to the Riccati equation
\eqref{RicA} if and only if the graph $\cG(X)$ of $X$,
\begin{equation}
\label{cGX}
 \cG(X):=\left\{x\oplus Xx\,\,|\,\, x\in\fH_{_A}\right\},
\end{equation}
is an invariant subspace of $L$. Similarly, a bounded operator $Y:\,\fH_{_D}\to\fH_{_A}$
is a solution to the Riccati equation \eqref{RicD}
if and only if the graph subspace
\begin{equation}
\label{cGY}
\cG(Y)=\left\{Yy\oplus y\,\,|\,\, y\in\fH_{_D}\right\}
\end{equation}
is invariant under $L$. Thus, the problem of existence and
uniqueness of solutions to the Riccati equations turns out to be an
important issue in various sections of mathematics and physics
involving the study of invariant subspaces of a linear operator.
Among them one may place the long-standing problem of obtaining
optimal bounds on variation of a spectral subspace of a self-adjoint
operator under an additive perturbation that still has only partial
solutions (see, e.g., the articles, in chronological order,
\cite{KMM02,AM-CAOT,SeelMain,Seel2016}, and references therein). It
is the possibility to construct reducing subspaces of a
quantum-mechanical Hamiltonian in terms of solutions to operator
Riccati equations that lies behind the celebrated Okubo \cite{Okubo}
and Foldy-Wouthoysen \cite{FW} transforms. Operator Riccati
equations and invariant graph subspaces are also closely related to
the factorization problem \cite{MarkusMatsaev75} for operator
pencils with resolvent-like dependence on the spectral parameter
(see \cite{AL,AdLT,ALMSr,LMMT,MS}).

Most of the known results on the solvability of the operator
Riccati equations \eqref{RicA}\,/\,\eqref{RicD} concern the case where
the spectra $\sigma_{_A}$ and $\sigma_{_D}$ of the main-diagonal entries
are disjoint, that is,
\begin{equation}
\label{dIntro}
d:=\dist\bigl(\sigma_{_A},\sigma_{_D}\bigr)>0,
\end{equation}
and the corresponding  block operator matrix $L$ is self-adjoint.
The total list of works touching the problem of the existence of
solutions to \eqref{RicA}\,/\,\eqref{RicD} associated with a
self-adjoint $L$ is rather extensive and here we mention only a very
few of the related publications:
\cite{AdLT,ALMSr,AMM,AM-CAOT,KMM02,MSS,MS,MotRem,Seel2016}. In the
case of a self-adjoint $L$, for certain mutual positions of the
(disjoint) spectral sets $\sigma_{_A}$ and $\sigma_{_D}$, even some
sharp conditions on $B$ (and $C=B^*$) ensuring the solvability of
\eqref{RicA}\,/\,\eqref{RicD} are available. These particular
spectral situations correspond to the mutual positions where one of
the of the spectral sets $\sigma_{_A}$ and $\sigma_{_D}$ is
completely embedded into a finite or infinite spectral gap of the
other set (see \cite{KMM03,KMM04}). The optimal solvability
conditions are accompanied by sharp norm bounds on the solution $X$
that follow from the relevant estimates in the subspace perturbation
problem known as the Davis-Kahan $\tan2\Theta$ theorem \cite{DK70}
and the \emph{a priori} $\tan\Theta$ theorem \cite{AM-TanT,MS01}.
Best available sufficient condition for the existence of a bounded
solution $X$ to \eqref{RicA} and best (but still not optimal) norm
estimate on $X$ under the single spectral assumption \eqref{dIntro}
follow from the main result of \cite{Seel2016} (cf.
\cite{NTSE-2014}). A number of the existence results for
\eqref{RicA} and estimates on the solution $X$ under the condition
\eqref{dIntro} in the case of a $J$-self-adjoint block operator
matrix $L$ may be found in \cite{AMSh,AMT} (also see
\cite{Veselic1,Veselic2}). Furthermore, we refer to \cite{AM2011}
concerning the existence results for \eqref{RicA} with disjoint
$\sigma_{_A}$ and $\sigma_{_D}$ in the case where one of the entries
$A$ and $D$ is a normal operator. Finally, in the generic
non-self-adjoint case, an existence result for the Riccati equation
\eqref{RicA} under condition \eqref{dIntro} has been obtained in
\cite{LMMT}, based the concept of quadratic numerical range.

In paper \cite{MM99} that treats the case of self-adjoint $L$, the
assumption \eqref{dIntro} is replaced by the hypothesis that the
spectrum of one of the main-diagonal entries $A$ and $D$ is at least
partly embedded into the absolutely continuous spectrum of the other
one, say
\begin{equation}
\label{MMincl}
\sigma_{_A}\cap\sigma_{_D}^{\rm ac}\neq\emptyset,
\end{equation}
where $\sigma_{_D}^{\rm ac}$ denotes the absolutely continuous
spectrum of $D$. Following the quantum-mechanical terminology, we
call the spectral disposition \eqref{MMincl} the \emph{Feshbach
case} since for infinitesimally small $B\neq 0$ (and $C=B^*$) the
eigenvalues of $A$ embedded into $\sigma_{_D}^{\rm ac}$ generically
transform into Feshbach resonances \cite{Feshbach}.

Conditions on the entry $B$ (and, hence, on the
entry $C=B^*$) in \cite{MM99} are chosen such that the Schur
complement
\begin{equation}
\label{MA}
{M_{_A}(z)=A-z-B(D-z)^{-1}C},\quad z\in\bbC\setminus\sigma_{_D},
\end{equation}
considered as an operator-valued function of $z$, admits analytic
continuation through bands of $\sigma_{_D}^{\rm ac}$
onto certain adjacent domains lying already on unphysical sheets of
the complex plane. It was found in \cite{MM99} that the continued
Schur complement \eqref{MA} admits a factorization of the
Markus-Matsaev type \cite{MarkusMatsaev75} and, thus, it possesses a
family of operator roots. The spectrum of an operator root of
$M_{_A}$, along with a part of the usual spectrum of $L$, may
possibly include a number of resonances of $L$. In \cite{HMM-JOT}
the results of \cite{MM99} were generalized to some unbounded
self-adjoint $L$ with unbounded $B$ and in \cite{HMM} even to some
unbounded non-self-adjoint $L$. Recently, in \cite{AM2016-1}, the
factorization approach of \cite{MM99} allowed us to prove the
existence of bounded solutions to the operator Riccati equation
\eqref{RicA} associated with a $J$-self-adjoint block matrix $L$ of
the form \eqref{L} in the case where
$\sigma_{_A}\subset\sigma_{_D}^{\rm ac}$.

In the present work we consider the case where the entries $A$ and
$D$ are self-adjoint, with $D$ being given in the spectral
representation. Thus, finally we even adopt the hypothesis that $D$
is simply the operator of multiplication by an independent variable.
Moreover, in order to ensure the maximal clarity, we then restrict
ourselves to the case where all the spectrum of $D$ is absolutely
continuous and uniform, being presented by a single band, that is,
$\sigma_{_D}=\sigma_{_D}^{\rm ac}=[\alpha,\beta]$,
$-\infty<\alpha<\beta<\infty$, and
$\sigma_{_A}\subset(\alpha,\beta)$. Therefore, the operator $L$ we
study, is in fact nothing but an extension of one of the two
celebrated Friedrichs models in \cite{Fried}, namely the $2\times2$
operator matrix model discussed in \cite[Section 6]{Fried}.
Furthermore, we assume that the entries $B$ and $C$ are defined via
operator-valued functions $\fb(\lambda)$ and $\fc(\lambda)$ of
$\lambda\in(\alpha,\beta)$ that are both real analytic and admit
analytic continuation onto some domain $\cD\subset\bbC$ (see
Section~\ref{Sscal} for details). This allows one to perform a
\emph{complex deformation} of $L$. The latter involves, in
particular, the replacement of the original entry $D$ with the
operators $D_{_\Gamma}$ of multiplication by the complex variable
$\lambda$ running through piecewise smooth Jordan contours $\Gamma$
obtained from the interval $(\alpha,\beta)$ by a continuous
transformation. It is assumed that during such a transformation the
end points $\alpha$ and $\beta$ are fixed and
$\Gamma\setminus\{\alpha,\beta\}\subset\cD$. For the complexly
deformed operators $B$ and $C$ we use the respective notations
$B_{_\Gamma}$ and $C_{_\Gamma}$. Notice that, in case of momentum
space few-body Hamiltonians, the approach we apply to $L$ is well
known under the name of \emph{contour deformation method} (see,
e.g., \cite{HVHJ} and references therein). One of variants of this
method that reduces the deformation of the absolutely continuous
spectrum just to its rotation in $\bbC$ about the threshold points
is the celebrated \emph{complex scaling}, used both in momentum and
coordinate representations (see \cite{BCombes,Lovelace,RS-III}).

The complex deformation of $L$ leads to the \emph{complexly
deformed} associated Riccati equations \eqref{RicA}\,/\,\eqref{RicD}
with the same entry $A$ but with $B$, $C$, and $D$ replaced by the
corresponding complexly deformed $B_{_\Gamma}$, $C_{_\Gamma}$, and
$D_{_\Gamma}$. The complexly deformed main-diagonal entry
$D_{_\Gamma}$ is a normal operator whose spectrum
$\sigma_{_{D_\Gamma}}=\overline{\Gamma}$ may be made disjoint with
$\sigma_{_A}$ by a relevant choice of the contour $\Gamma$. Then one
simply applies to the deformed Riccati equations the approach of
\cite{AM2011} that works under the assumption of spectral
disjointness \eqref{dIntro} and that we already mentioned above. In
particular, we prove that the operator roots of the continued Schur
complement \eqref{MA} are explicitly expressed through the solutions
$X_{_\Gamma}$ to the complexly deformed Riccati equation
\eqref{RicA}.

The article is organized as follows. In Section \ref{SPrel} we
collect the necessary information on the existence and properties of
solutions to the Riccati equations \eqref{RicA}\,/\,\eqref{RicD}
with special attention to the case where at least one of the entries $A$ and
$D$ is a normal operator.  In Section \ref{sFactor}, we present a
version of results of \cite{MM99} adapted to the case
$\sigma_{_A}\subset\sigma_{_D}^{\rm ac}$ under consideration.
However, unlike in \cite{MM99}, we do not require that $C=B^*$.
Among other things, the section contains conditions ensuring the
existence of operator roots for the analytically continued Shur
complement \eqref{MA}. Finally, in Section \ref{Sscal} we introduce
an extension of the Friedrichs' $2\times2$ operator matrix model
from \cite[Section 6]{Fried} and consider its variant admitting a
complex deformation. Assuming the existence of a piecewise smooth
Jordan contour $\Gamma$ such that
$\Gamma\setminus\{\alpha,\beta\}\subset\cD\cap\bbC^\pm$ and the
norms\footnote{See Definition \ref{DefEN} below for the norm of a
bounded operator with respect to the spectral measure of a normal
operator.} $\|B_{_{\Gamma}}\|_{\sE_{_{D_\Gamma}}}$ and
$\|C_{_{\Gamma}}\|_{\sE_{_{D_\Gamma}}}$ of the deformed entries
$B_{_\Gamma}$ and $C_{_\Gamma}$ with respect to the spectral measure
of the normal operator $D_{_\Gamma}$ satisfy the condition
\begin{equation}
\label{BCED2dIn} \sqrt{\|B_{_{\Gamma}}\|_{\sE_{_{D_{\Gamma}}}}
\|C_{_{\Gamma}}\|_{\sE_{_{D_\Gamma}}}}<\frac{1}{2}
\dist(\sigma_{_A},\Gamma),
\end{equation}
we prove the existence of bounded solutions $X_{_\Gamma}$ and
$Y_{_\Gamma}$ to the complexly deformed Riccati equations
\eqref{RicA} and \eqref{RicD} (see Theorem \ref{ThFin}). The
operator roots of the analytically continued Schur complement
\eqref{MA} are nothing but the operators
$Z_{_A}=A+B_{_{\Gamma}}X_{_{\Gamma}}$. Under \eqref{BCED2dIn} these
operators depend only on the sign $\fs=\pm$ in the half-plane
$\bbC^\fs$ superscript but not on the (form of the) contour
$\Gamma\subset\cD\cap(\bbC^\fs\cup\bbR)$ itself. The solutions
$X_{_\Gamma}$ and $Y_{_\Gamma}$ possess the property $\|X_{_\Gamma}
Y_{_\Gamma}\|<1$ which guarantees the block diagonalizability of the
complexly deformed operator matrix $L_{_\Gamma}$ (see Corollary
\ref{CorDiag}).

The following notations are used thro\-ug\-h\-o\-ut the paper. By a
subspace of a Hilbert space we always understand a closed linear
subset. The identity operator is denoted by $I$. The Banach space of
bounded linear operators from a Hilbert space $\fL$ to a Hilbert
space $\fM$ is denoted by $\cB(\fL,\fM)$ and by $\cB(\fL)$ if
$\fL=\fM$. By $\sigma_{_S}$ we denote the spectrum of an operator
$S\in\cB(\fM)$. The notation $\sE_T(\delta)$ is used for the
spectral projection of a normal operator $T$ associated with a Borel
set $\delta\subset\bbC$. In the particular case where $T$ is
self-adjoint, $\delta\subset\bbR$. By $\overline{\delta}$ we denote
the closure of an arbitrary $\delta\subset\bbC$.  By
$\cO_r(\delta)$, $r>0$, we denote the open $r$-neigh\-bourhood of
$\delta$ in $\bbC$, i.e.\,
$\cO_r(\delta)=\{z\in\bbC\big|\,\dist(z,\delta)< r\}$. By $\bbC^+$
and $\bbC^-$ we understand respectively the upper and lower
half-planes of the complex plane $\bbC$ (with excluded real axis),
that is, $\bbC^\pm=\{z\in\bbC\,|\,\pm\mathop{\rm Im} z>0\}$.
\vspace*{2mm}

\noindent {\bf Acknowledgements.} A.K.M is indebted to the Institut
f\"ur Angewandte Mathematik for kind hospitality during his stays at
the Universit\"at Bonn in 2016 and 2017. This work was supported by
the Heisenberg-Landau Program, the Deutsche
For\-sch\-ungs\-gemeinschaft (DFG), and the Russian Foundation for
Basic Research.

\section{Preliminaries}
\label{SPrel}

Assume that the bounded operators $X\in\cB(\fH_{_A},\fH_{_D})$ and
$Y\in\cB(\fH_{_D},\fH_{_A})$ are solutions to the operator Riccati
equations \eqref{RicA} and \eqref{RicD}, respectively. The operators
\begin{align}
\label{Za}
Z_{_A}&=A+BX,\\
\label{Zd}
Z_{_D}&=D+CY,
\end{align}
and
\begin{align}
\label{tZa}
\widetilde{Z}_{_A}&=A-YC,\\
\label{tZd}
\widetilde{Z}_{_D}&=D-XB,
\end{align}
play an outstanding role in the spectral theory of the block
operator matrices of the form \eqref{L} and related operator
pencils. This concerns, in particular, the Schur complements $M_{_A}$
and $M_{_D}$ corresponding to the matrix $L$, $M_{_A}(z)$ is given by
\eqref{MA} and
\begin{equation}
\label{MD}
M_{_D}(z)=D-z-C(A-z)^{-1}B,\quad z\in\bbC\setminus\sigma_{_A}.
\end{equation}
One easily verifies by inspection that the following identities
hold:
\begin{equation}
\label{roots} {M_{_A}(z)=W_{_A}(z){(Z_{_A}-z)}},\quad
z\in\bbC\setminus\sigma_{_D},\quad\text{ and }\quad
{M_{_D}(z)=W_{_D}(z){(Z_{_D}-z)}},\quad
z\in\bbC\setminus\sigma_{_A},
\end{equation}
where $Z_{_A}$ and $Z_{_D}$ are the operators \eqref{Za} and
\eqref{Zd}, respectively; the entries $W_{_A}$ and $W_{_D}$
are explicitly given by
\begin{equation}
\label{WAD}
W_{_A}(z)=I-B(D-z)^{-1}X\qquad\text{ and }\qquad W_{_D}(z)=I-C(A-z)^{-1}Y.
\end{equation}
Similarly,
\begin{align}
\label{troots}
M_{_A}(z)&=(\widetilde{Z}_{_A}-z)\widetilde{W}_{_A}(z),\quad
z\in\bbC\setminus\sigma_{_D},\quad \text{and}\quad
M_{_D}(z)=(\widetilde{Z}_{_D}-z)\widetilde{W}_{_D}(z),\quad
z\in\bbC\setminus\sigma_{_A}
\end{align}
where $\widetilde{Z}_{_A}$ and $\widetilde{Z}_{_A}$ are defined by \eqref{tZa} and \eqref{tZd}, and
\begin{equation}
\label{tWAD}
\widetilde{W}_{_A}(z)=I+Y(D-z)^{-1}C, \quad
\widetilde{W}_{_D}(z)=I+X(A-z)^{-1}B.
\end{equation}

If it so happened that
\begin{equation}
\label{XYn1}
1\not\in\spec(XY)\quad \text{and, equivalently,}\quad 1\not\in\spec(YX),
\end{equation}
the off-diagonal block operator matrix
\begin{equation}
Q=\left(\begin{array}{cc}
0 & Y \\
X & 0
\end{array}\right)
\end{equation}
composed of the solutions $X$  and $Y$ allows one to perform
similarity transformations of the operator $L$ into block diagonal
operator matrices formed either of the operators \eqref{Za}, \eqref{Zd} or
operators \eqref{tZa}, \eqref{tZd}. Namely,
\begin{equation}
\label{Ldiag}
L=(I+Q)\left(\begin{array}{cc}
Z_{_A} & 0 \\
0 & Z_{_D}
\end{array}\right)(I+Q)^{-1}=(I-Q)^{-1}\left(\begin{array}{cc}
\widetilde{Z}_{_A} & 0 \\
0 & \widetilde{Z}_{_D}
\end{array}\right)(I-Q),
\end{equation}

Under condition \eqref{XYn1}, from \eqref{Ldiag} it follows that the operators $Z_{_A}$
and $\widetilde{Z}_{_A}$ as well as the operators $Z_{_D}$
and $\widetilde{Z}_{_D}$ are pairwise similar to each other. More precisely,
\begin{align}
\label{ZZa}
\widetilde{Z}_{_A}&=(I-YX)Z_{_A}(I-YX)^{-1},\\
\label{ZZd}
\widetilde{Z}_{_D}&=(I-XY)Z_{_D}(I-XY)^{-1}.
\end{align}

If the operator $D$ is normal and the spectra of $D$ and $Z_{_A}$
are disjoint, the solution $X$ admits the following integral
representation (see \cite{AM2011} for the proof and definition of
the integral over the spectral measure involved):
\begin{equation}
\label{XZ}
{X=\int_{\sigma_{_D}} \sE_{_D}(d\mu)C({Z_{_A}}-\mu)^{-1}},
\end{equation}
where $\sE_{_D}$ is the spectral measure of $D$.
This representation, written in the form
\begin{equation}
\label{Xeq}
X=\int_{\sigma_{_D}} \sE_{_D}(d\mu)C(A+BX-\mu)^{-1},
\end{equation}
may be treated as one more equation for determining~$X$.
Similarly, the disjointness of the spectra of $D$ and $\widetilde{Z}_{_A}=A-YC$
yields an ``integral equation'' for $Y$,
\begin{equation}
\label{YZ}
{Y=-\int_{\sigma_{_D}} (A-YC-\mu)^{-1}}B \sE_{_D}(d\mu).
\end{equation}
Notice that \eqref{XZ} allows one to rewrite the function $W_{_A}(z)$ from \eqref{WAD}
in the form
\begin{equation}
\label{WADx}
W_{_A}(z)=I-\int_{\sigma_{_D}} B\sE(d\mu)C\frac{1}{\mu-z}(Z_{_A}-\mu)^{-1}.
\end{equation}

The paper \cite{MM99} introduced the concept of the norm of a
bounded operator with respect to the spectral measure of a given self-adjoint
operator. In \cite{AM2011} this concept was extended to the spectral
measure associated with a given normal operator. The operator norm with
respect to a spectral measure proved to be a useful tool in the
study of the operator Sylvester and Riccati equations (see
\cite{AMM} and \cite{AM2011} for details). We recall this concept
bearing in mind its application to equations \eqref{Xeq} and
\eqref{YZ}.

\begin{definition}
\label{DefEN}
\label{ENorm} Let $S\in\cB(\fH_{_{A}},\fH_{_D})$ be a bounded
operator between the Hilbert spaces $\fH_{_{A}}$ and $\fH_{_{D}}$, and
let the operator $D\in\cB(\fH_{_D})$ be normal.
Introduce the quantity
\begin{equation}
\label{enorma}
\|S\|_{\sE_D}=\left(\sup\limits_{\{\delta_j\}}
\sum_j \|S^*\sE_{_D}(\delta_j)S\|\right)^{1/2},
\end{equation}
where the supremum is taken over finite (or countable) systems of
mutually disjoint Borel subsets $\delta_j$ of the spectrum
$\sigma_{_D}$ of the normal operator $D$, $\delta_j\cap\delta_k=\emptyset$, if $j\neq k$.  The
number $\|S\|_{E_D}$ is called the norm of $S$ with
respect to the spectral measure $d\sE_{_D}(z)$ or simply $\sE_{_D}$-norm of $S$.
For $T\in\cB(\fH_{_{D}},\fH_{_A})$ the $E$-norm $\|T\|_{\sE_D}$ is defined by
$\|T\|_{\sE_D}:=\|T^*\|_{\sE_D}$.
\end{definition}
\begin{remark}
Clearly, Definition \ref{DefEN} implies
\begin{equation}
\label{STEN}
\|S\|\leq \|S\|_{E_D}\quad\text{and}\quad \|T\|\leq \|T\|_{E_D}.
\end{equation}
\end{remark}

In the case where both the operators $A$ and $D$ are normal one is
able to prove the existence of fixed points for the mappings on the
right-hand sides of \eqref{Xeq} and \eqref{YZ} provided that the
operators  $B$ and $C$ satisfy certain smallness conditions
involving the $\sE_{_D}$-norms of $B$ or $C$ (see \cite{AM2011}; for earlier
results with self-adjoint $A$ and/or $D$ see \cite{AdLT,AMM,MotRem}).

\begin{remark}
For the special case of $B=0$, the Riccati equation \eqref{RicA} turns into a linear equation
\begin{equation}
\label{SylvA}
XA-DX=C, \quad X\in\cB(\fH_{_A},\fH_{_D}),
\end{equation}
called the Sylvester equation. Similarly, for $C=0$, the Riccati equation \eqref{RicD}
turns into the Sylvester equation
\begin{equation}
\label{SylvD}
YD-AY=B, \quad Y\in\cB(\fH_{_D},\fH_{_A}).
\end{equation}
If the entry $D$ is a normal operator and $\sigma_{_A}\cap\sigma_{_D}=\emptyset$,
the unique bounded solutions $X$ and $Y$ to \eqref{SylvA} and \eqref{SylvD} are given, respectively, by
\begin{equation}
\label{SolSyl}
X=\int_{\sigma_{_D}} \sE_{_D}(d\mu)C(D-\mu)^{-1}\quad\text{and}\quad
Y=-\int_{\sigma_{_D}} (A-\mu)^{-1}B \sE_{_D}(d\mu)
\end{equation}
(cf. \eqref{XZ} and \eqref{YZ}; see \cite[Theorem 4.5]{AM2011}).
\end{remark}

The following statement is a particular case of
\cite[Theorem~5.7]{AM2011}.

\begin{theorem}
\label{QsolvN} Let both operators $A\in\cB(\fH_{_A})$ and
$D\in\cB(\fH_{_D})$ in \eqref{RicA} be normal. Assume that $0\neq
B\in\cB(\fH_{_A},\fH_{_D})$ and
\begin{equation}
\label{tsemibd}
  d=\dist\bigl(\sigma_{_A},\sigma_{_D}\bigr)>0.
\end{equation}
Also assume that the operator $C\in\cB(\fH_{_D},\fH_{_A})$ has a
finite $\sE_{_D}$--norm and
\begin{equation}
\label{BCest}
\sqrt{\|B\|\,\|C\|_{_{\sE_{_D}}}}<\frac{d}{2}.
\end{equation}
Then the Riccati equation \eqref{RicA} has a unique solution $X$
in the ball
\begin{equation}
\label{QEst2}
\left\{T\in\cB(\fH_{_A},\fH_{_D})\,\big|\,\,\,
\|T\|<\|B\|^{-1}\left(d-\sqrt{\|B\|\,\|C\|_{_{\sE_D}}}\right)\right\}.
\end{equation}
Moreover, the solution $X$ has a finite $\sE_{_D}$--norm that satisfies
the bound
\begin{equation}
\label{QEst1}
\|X\|_{\sE_{_D}}\leq
 \frac{1}{\|B\|}\,
\left(\frac{d}{2}-\sqrt{\frac{d^2}{4}-\|B\|\,\|C\|_{_{\sE_D}}}\right).
\end{equation}
\end{theorem}

A similar statement concerns the Riccati equation \eqref{RicD}.

\begin{theorem}
\label{QsolvN1} Let both the operators $A\in\cB(\fH_{_A})$ and
$D\in\cB(\fH_{_D})$ in \eqref{RicD} be normal. Assume that $0\neq
C\in\cB(\fH_{_D},\fH_{_A})$ and condition \eqref{tsemibd} holds.
Assume in addition that the operator $B\in\cB(\fH_{_A},\fH_{_D})$
has a finite $\sE_{_D}$--norm and
\begin{equation}
\label{BCest1}
\sqrt{\|B\|_{_{\sE_{_D}}}\|C\|}<\frac{d}{2}.
\end{equation}
Then the Riccati equation \eqref{RicD} has a unique solution $Y$
in the ball
\begin{equation}
\label{QEst3}
\left\{S\in\cB(\fH_{_D},\fH_{_A})\,\big|\,\,\,
\|S\|<\|C\|^{-1}\left(d-\sqrt{\|B\|_{_{\sE_D}}\,\|C\|}\right)\right\}.
\end{equation}
Moreover, the solution $Y$ has a finite $\sE_{_D}$--norm that satisfies
the bound
\begin{equation}
\label{QEst4}
\|Y\|_{_{\sE_{_D}}}\leq
 \frac{1}{\|C\|}\,
\left(\frac{d}{2}-\sqrt{\frac{d^2}{4}-\|B\|_{_{\sE_{_D}}}\|C\|}\right).
\end{equation}
\end{theorem}
\begin{corollary}
\label{CorN1}
Assume condition
\begin{equation}
\label{BCED2}
\sqrt{\|B\|_{_{\sE_{_D}}}\|C\|_{_{\sE_{_D}}}}<\frac{d}{2}.
\end{equation}
Under this condition the following inequalities hold:
\begin{equation}
\label{XYl1}
\|X\|_{_{\sE_{_D}}}\|Y\|_{_{\sE_{_D}}}\leq\frac{\|B\|_{_{\sE_{_D}}}\|C\|_{_{\sE_{_D}}}}{d^2/4}<1.
\end{equation}
\end{corollary}
\begin{proof}
Notice that due to \eqref{STEN} the bound \eqref{BCED2} implies both
the estimates \eqref{BCest} and \eqref{BCest1}. Hence, the existence
of solutions $X$ and $Y$ satisfying the corresponding bounds
\eqref{QEst1} and \eqref{QEst4} is ensured.
For the right-hand sides of these bounds we have
\begin{align}
\label{rhs1}
\frac{1}{\|B\|}\,
\left(\frac{d}{2}-\sqrt{\frac{d^2}{4}-\|B\|\,\|C\|_{_{\sE_D}}}\right)&=
\frac{\|C\|_{_{\sE_D}}}{\frac{d}{2}+\sqrt{\frac{d^2}{4}-\|B\|\,\|C\|_{_{\sE_D}}}}< \frac{\|C\|_{_{\sE_D}}}{d/2},\\
\label{rhs2}
\frac{1}{\|C\|}\,
\left(\frac{d}{2}-\sqrt{\frac{d^2}{4}-\|B\|_{_{\sE_{_D}}}\|C\|}\right)&=
\frac{\|B\|_{_{\sE_{_D}}}}{\frac{d}{2}+\sqrt{\frac{d^2}{4}-\|B\|_{_{\sE_{_D}}}\|C\|}}<
\frac{\|B\|_{_{\sE_D}}}{d/2}.
\end{align}
Then \eqref{QEst1} and \eqref{QEst4} together with \eqref{BCED2} imply \eqref{XYl1}.
\end{proof}
\begin{remark}
\label{Rem1} Taking into account \eqref{STEN}, from the bound
\eqref{XYl1} it follows that the products $XY$ are $YX$ are strict
contractions $\|XY\|<1$ and $\|YX\|<1$, which means that under the
condition \eqref{BCED2} the block operator matrix \eqref{L} is block
diagonalizable in any of the two forms \eqref{Ldiag}.
\end{remark}

Now consider the operators $Z_A$ and $Z_D$ built according
\eqref{Za} and \eqref{Zd} of the corresponding unique solutions $X$
and $Y$ referred to in Theorems \ref{QsolvN} and \ref{QsolvN1}. In
particular, under the conditions of Theorem \ref{QsolvN} the
operator-valued function $W_{_A}(z)$, introduced in \eqref{WAD}, is
boundedly invertible and holomorphic in $z$ at least on the open
$d/2$-neighborhood $\cO_{d/2}(\sigma_{_A})$ of the set
$\sigma_{_A}$. By \eqref{QEst1} this neighborhood contains the whole
spectrum of $Z_{_A}$. Analogously, under the hypothesis
of Theorem \ref{QsolvN1} the operator-valued function $W_{_D}(z)$ is
boundedly invertible and holomorphic in $z$ at least on the open
$d/2$-neighborhood $\cO_{d/2}(\sigma_{_D})$ of the set $\sigma_{_D}$
that contains the whole spectrum of $Z_{_D}$. In such a case, the
factorization \eqref{roots} yields that $Z_A$ and $Z_D$ are nothing
but the \emph{operator roots} of the Schur complements $M_{_A}(z)$
and $M_{_D}(z)$ in the sense of Markus-Matsaev
\cite{MarkusMatsaev75}. From \eqref{roots} it follows
that $\spec(M_{_A})\cap\cO_{d/2}(\sigma_{_A})=\spec(Z_{_A})$ and
$\spec(M_{_D})\cap\cO_{d/2}(\sigma_{_A})=\spec(Z_{_D})$.

The same consideration is relevant to the operators
$\widetilde{Z}_A$ and $\widetilde{Z}_D$ defined respectively by
\eqref{tZa} and \eqref{tZd} provided that $X$ and $Y$ are again the
unique solutions mentioned in Theorems \ref{QsolvN} and
\ref{QsolvN1}. In the sense of the factorizations \eqref{troots},
these operators \eqref{tZa} and \eqref{tZd} may be named the
\emph{left operator roots} of the Schur complements  $M_{_A}$ and
$M_{_D}$, respectively (cf. \cite[Theorem 4.1]{HMM}).

\section{Factorization of one of the Schur complements in the Feshbach case}
\label{sFactor}

From now on we assume that the entries $A$ and $D$ are self-adjoint
operators. It is also supposed that the spectra of $A$
and $D$ overlap. More precisely, we want to consider the situation
where at least a part of the spectrum of $A$ lies on the absolutely
continuous spectrum of $D$. There are examples (see \cite[Remark 3.9
and Lemma 3.10]{AMM}) which show that, in  such a spectral
situation, (conventional) solutions to the associated Riccati
equations may not exist at all.

Nevertheless, one can think of the Markus-Matsaev factorization
\cite{MarkusMatsaev75} and operator roots of the analytically
continued Schur complements. This idea has been fist elaborated  in
\cite{MM99} for self-adjoint block operator matrices $L$ involving
bounded off-diagonal entries. Later on, the approach of  \cite{MM99}
has been extended in \cite{HMM-JOT} and \cite{HMM} to some unbounded
off-diagonal entries in the respective cases of self-adjoint and
non-self-adjoint $L$.

\begin{figure}[htb]
\qquad\includegraphics[angle=0,width=12.cm]{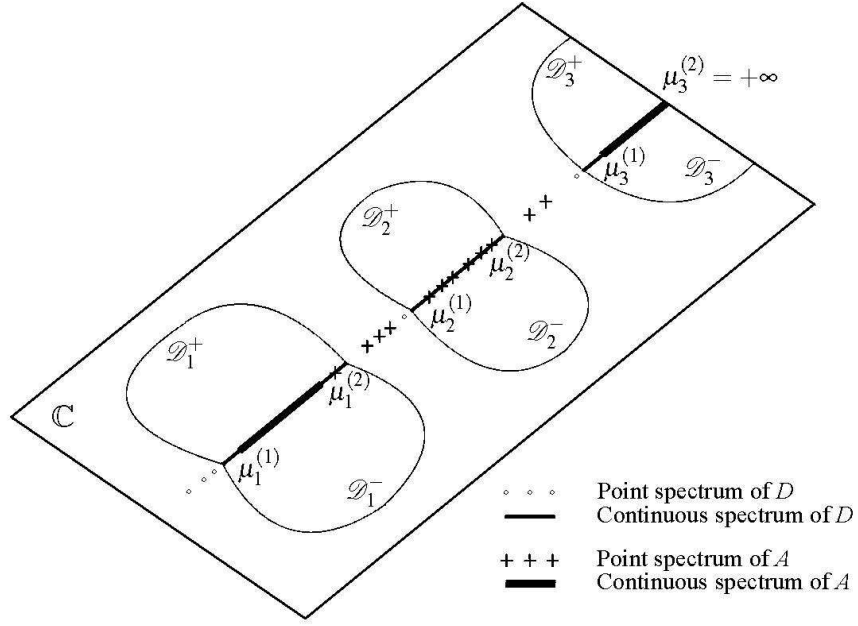}
\caption{An example of the spectral situation considered in
\cite{MM99} with a self-adjoint operator $D$ having three disjoint
intervals of the absolutely continuous spectrum
$[\mu_1^{(j)},\mu_2^{(j)}]$, $j=1,2,3$. In \cite{MM99} it is assumed
that $C=B^*$ and the entries $D$ and $B$ are such that the Schur
complement $M_A(z)$ admits analytic continuation in $z$ through the
intervals $[\mu_1^{(j)},\mu_2^{(j)}]\subset\sigma_{_D}^{\rm ac}$ into
certain domains $\cD^\pm_j$ (lying already on the unphysical
sheets of the Riemann surface of $M_{_{A}}$).} \label{Fig1}
\end{figure}

In order to recall the idea of the approach \cite{MM99}, let us
rewrite the Schur complement \eqref{MA} in terms of the spectral
measure $\sE_{_D}$ of the self-adjoint operator $D$:
\begin{align}
\label{MAE}
M_{_A}(z)&=A-z-B\int_{\sigma_{_D}} \sE_{_D}(d\mu)\frac{1}{\mu-z}C\\
\label{MAE1}
&=A-z-\int_\bbR B\,d\sE^{^D}(\mu)C\frac{1}{\mu-z},
\end{align}
where {$\sE^{^D}(\mu)=\sE_{_D}\bigl((-\infty,\mu)\bigr)$} is the
spectral function of $D$. In \cite{MM99} it is assumed (for $C=B^*$)
that the entries $D$ and $B$ are such that the operator-valued
function $M_A(z)$ admits analytic continuation in $z$ through the
appropriate segments of $\sigma_{_D}^{\rm ac}$ to certain domains
located on the so-called {\textit{unphysical} sheets} of the {$z$}
plane (see Figure \ref{Fig1}; this figure is borrowed from
\cite{MM99}, to which we also refer for the concept of unphysical
sheet).

To make the presentation as clear as possible, we reduce the
consideration to the case where all the spectrum of $D$ consists of
a single finite interval of the absolutely continuous spectrum, that
is,
\begin{equation}
\label{SDac}
\sigma_{_D}=\sigma_{_D}^{\rm ac}=\overline{\Delta},
\end{equation}
where
\begin{equation}
\label{Delta}
\Delta=(\alpha,\beta)\quad\text{for some}\quad\alpha,\beta\in\bbR,\quad \alpha<\beta,
\end{equation}
and
\begin{equation}
\label{Feshbach}
\sigma_{_A}\subset\Delta.
\end{equation}

\begin{figure}[htb]
\centering
\includegraphics[angle=-1.5,width=7.cm]{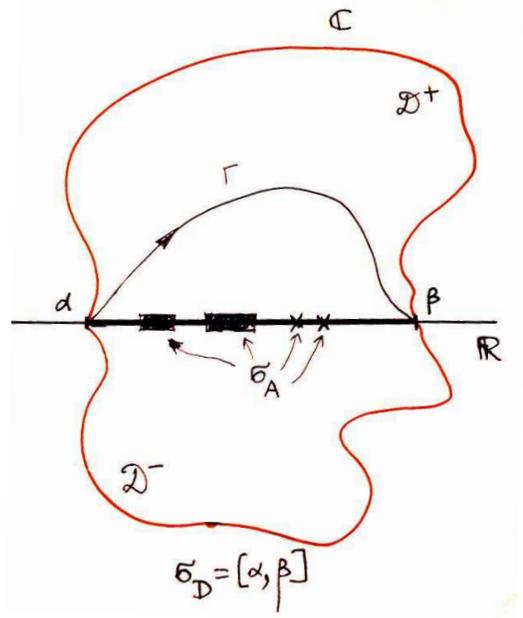}
\caption{The spectra of the operators $A$ and $D$, holomorphy
domains $\cD^-$ and $\cD^+$ of the continued operator-valued
function $K$, and the integration contour $\Gamma$.} \label{Fig2}
\end{figure}

\noindent Main hypothesis (see \cite{MM99}) is that the $\cB(\fH_{_A})$-valued
function
\begin{equation}
\label{Kdef}
K(\mu):=B\sE^D(\mu)C, \quad \mu\in\bbR,
\end{equation}
is real analytic on the interval $\Delta$ and admits
analytic continuation from $\Delta$ onto a domain
$\cD\subset\bbC$. This assumption entails the inclusion
$\Delta\subset\cD$. By $\cD^-$ and $\cD^+$ we will denote
the parts\footnote{Notice that in the self-adjoint case with
$C=B^*$, the domains $\cD^-$ and $\cD^+$ are necessarily symmetric
with respect to the real axis, $\cD^+=(\cD^-)^*$, and
$K(\mu^*)=K(\mu)^*$.} of the domain $\cD$ lying respectively
in the lower and upper half-planes of $\bbC$,
$\cD^\pm=\cD\cap\bbC^\pm$ (see Figure \ref{Fig2}).  The derivative
$K'(\mu)=\frac{d}{d\mu}K(\mu)$, $\mu\in\cD$, is allowed to be weakly
singular at the points $\alpha$ and $\beta$, namely,
\begin{subequations}
    \label{Ksing1}
\begin{align}
\tag{\ref{Ksing1}\,\textit{a}}\label{Ksing1:a}
\|K'(\mu)\| & \leq c |\mu-\alpha|^{-\nu}\quad
\text{for any $\mu\neq\alpha$ lying in an open neighborhood of $\alpha$ with $\cD$},\\
\tag{\ref{Ksing1}\,\textit{b}}\label{Ksing1:b}
\|K'(\mu)\| & \leq c |\mu-\beta|^{-\nu} \quad
\text{for any $\mu\neq\beta$ lying in an open neighborhood of $\beta$ in $\cD$},
\end{align}
\end{subequations}
where $c$ and $\nu$ are some real constants, $c>0$ and $0\leq\nu<1$. All the above allows
one to rewrite \eqref{MAE} in the form
\begin{align}
\label{MAEK}
M_{_A}(z)&=A-z-\int_\alpha^\beta d\mu\frac{K'(\mu)}{\mu-z},
\end{align}
where the integral term  is well defined and holomorphic for
$z\in\bbC\setminus\overline{\Delta}$.

Suppose that $\Gamma^-$ is a piecewise smooth Jordan contour having
the end points $\alpha$ and $\beta$ and, except for these points, lying
totally in $\cD^-$. Similarly, the notation $\Gamma^+$ is used for a
piecewise smooth Jordan contour having the end points $\alpha$ and
$\beta$ and, except for $\alpha$ and $\beta$, lying completely in
$\cD^+$. For $\Gamma=\Gamma^-$ or $\Gamma=\Gamma^+$, by
$\Omega(\Gamma)$ we denote the domain lying inside the closed curve
formed by the interval $\Delta$ and contour $\Gamma$. Thus,
$\Omega(\Gamma^\pm)\subset\cD^\pm$.

The integral term on the right-hand side of \eqref{MAEK} is a Cauchy
type integral. Then it is elementary to prove that, under the
assumptions adopted in this section, the function $M_{_A}(z)$ admits
analytic continuation in $z$ across the interval $(a,b)$ both from
the bottom up and from the top down. After such a continuation one
arrives to the sheet(s) of the Riemann surface of the function
$M_{_A}$ that differs from the original sheet of the spectral
parameter plane. For $z\in\Omega(\Gamma)$ the corresponding
continuation of $M_{_A}$ is given by
\begin{equation}
\label{MAK}
M_{_A}(z,\Gamma):=A-z-\int_{\Gamma} d\mu \frac{K'(\mu)}{\mu-z}, \quad \Gamma=\Gamma^\pm.
\end{equation}
We note that as a function of $z\in\bbC\setminus\Gamma^\pm$, the mapping
$M_{_A}(\cdot,\Gamma^\pm)$ possesses the property (see \cite[Lemma 2.1]{MM99})
\begin{equation}
\label{M1Gresidue}
 M_{_A}(z,\Gamma^\fs)=\left\{\begin{array}{cl}
 M_{_A}(z), & z\in\bbC\setminus\overline{\Omega(\Gamma^\fs)},\\
 M_{_A}(z)+2\pi i\,\mathfrak{s}\, K'(z), & z\in \Omega(\Gamma^\fs),
\end{array}\right. \quad \fs=\pm 1.
\end{equation}
Here and hereafter we identify the number $\fs=+1$ or $\fs=-1$ in a superscript or subscript
with the corresponding sign in $\pm$, that is,  e.g.,
$\Gamma^{+1}\equiv\Gamma^+$ and $\Gamma^{-1}\equiv\Gamma^-$.

Now let us introduce the equation
\begin{equation}
\label{EqMain} Z_{_A}=A-\int_\Gamma d\mu\, K'(\mu)(Z_{_A}-\mu)^{-1},
\quad \Gamma=\Gamma^\pm, \quad \Omega(\Gamma^\pm)\subset\cD^\pm,
\end{equation}
that makes sense, of course, provided the spectrum $\sigma_{_{Z_A}}$ of
the unknown $Z_{_A}\in\cB(\fH_{_A})$ does not intersect the integration
contour $\Gamma$, i.e. if $\sigma_{_{Z_A}}\cap\Gamma=\emptyset$. Also, let us
introduce the quantity
\begin{equation}
\label{cVK}
\cV_{_K}(\Gamma):=\int_\Gamma |d\mu|\,\|K'(\mu)\|
\end{equation}
that we call the variation of the operator-valued function $K$ in
\eqref{Kdef} along the contour $\Gamma$, and let
\begin{equation}
\label{dsG}
d(\Gamma):=\dist(\sigma_{_A},\Gamma).
\end{equation}

Applying to \eqref{EqMain} Banach's Fixed Point Theorem results in
(cf. \cite[Theorem~3.1]{MM99}):

\begin{theorem}
\label{ThSolv1} Let $\Gamma^\fs$, $\fs=\pm1$, be a piecewise smooth
Jordan contour having the end points $\alpha$, $\beta$ and being
such that $\Omega(\Gamma^\pm)\subset\cD^\pm$. Assume that
\begin{equation}
\cV_K(\Gamma^\fs)<\,\frac{1}{4}\,d({\Gamma^\fs})^2.
\label{cVKd}
\end{equation}
Then the equation \eqref{EqMain} has a solution $Z_{_A}^\fs$ of the form
\begin{equation}
\label{ZAT}
Z_{_A}^\fs=A+T^\fs
\end{equation}
with
\begin{equation}
\label{Tmbound}
\|T^\fs\|\leq r(\Gamma^\fs),
\end{equation}
where
\begin{equation}
\label{Tmbound}
r(\Gamma^\fs)=\frac{d(\Gamma^\fs)}{2}-\sqrt{\frac{d(\Gamma^\fs)^2}{4}-\cV_{_K}(\Gamma^\fs)}.
\end{equation}
The solution $Z_{_A}^\fs$ of the form \eqref{ZAT} is unique in the closed
ball in $\cB(\fH_{_A})$ centered at zero and having the radius
$d(\Gamma^\fs)-\sqrt{\cV_{_K}(\Gamma^\fs)}$.
\end{theorem}

\begin{lemma}
\label{HuniqueCor} For a fixed value of $\fs$, $\fs=\pm1$, the unique solution
$Z_{_A}^\fs$ of the form \eqref{ZAT} referred to in Theorem~\ref{ThSolv1}
is the same for all the piecewise smooth Jordan contours $\Gamma^\fs$
having the end points $\alpha$, $\beta$ and being such that
\begin{equation}
\label{OcVKd}
\Omega(\Gamma^\fs)\subset\cD^\fs\quad\text{and}\quad
\cV_K(\Gamma^\fs)<\,\frac{1}{4}\,d({\Gamma^\fs})^2.
\end{equation}
Moreover, the following norm bound
holds
\begin{equation}
\label{Xr0}
\|Z_{_A}^\fs-A\|\leq r_0(K)
\end{equation}
with
\begin{equation}
\label{r0}
r_0(K):=\inf\limits_{\Gamma^\fs:\,\omega(\Gamma_l)>0}
r(\Gamma^\fs)
\end{equation}
where the infimum is taken over all piecewise smooth Jordan contours
$\Gamma_l$ satisfying \eqref{OcVKd}, the quantity $r(\Gamma^\fs)$ is
given by \eqref{Tmbound}, and
\begin{equation}
\label{omega}
\omega(\Gamma_l)=d_0^2(\Gamma^\fs)-4\cV_{_K}(\Gamma_l).
\end{equation}
\end{lemma}

\noindent The proof of Lemma \ref{HuniqueCor}  almost literally repeats the proof of
Corollary 3.4 in \cite{MM99}. Thus, we omit it, too, as well as the proof of the following

\begin{corollary}
The spectrum of $Z_{_A}^\fs$ lies in the closed complex
$r_0(K)$-neighborhood $\overline{\cO_{r_0(K)}(\sigma_{_A})}$ of the
spectrum $\sigma_{_A}$ of the operator $A$.
\end{corollary}

We conclude the section by presenting a factorization result for
$M_{_A}(\cdot,\Gamma^\pm)$. We again skip proof since it follows
exactly the same line as the proof of Theorem 4.1 in \cite{MM99}.

\begin{lemma}
\label{LfactGA} Assume that the hypothesis of Theorem \ref{ThSolv1}
is satisfied and let $Z_{_A}^\fs$, $\fs=\pm1$, be the unique solution to
\eqref{EqMain} referred to in that theorem. Then for
$z\in\bbC\setminus\Gamma^\fs$ the operator-valued function
$M_{_A}(z,\Gamma^\fs)$ admits the following factorization:
\begin{equation}
\label{MAGfact}
M_{_A}(z,\Gamma^\fs)=W_{_A}(z,\Gamma^\fs)(Z_{_A}^\fs -z),
\end{equation}
where $W_{_A}(z,\Gamma^\fs)$ is given by
\begin{equation}
\label{WAzG}
W_{_A}(z,\Gamma^\fs)=I-\int_{\Gamma^\fs}d\mu\, K'(\mu)\frac{1}{\mu-z}(Z_{_A}^\fs-\mu)^{-1}.
\end{equation}
The operator $W_{_A}(z,\Gamma^\fs)$ is bounded, that is,
$W_{_A}(z,\Gamma^\fs)\in\cB(\fH_{_A})$, whenever
$z\in\bbC\setminus\Gamma^\fs$. Moreover, for $\dist(z,\sigma_A)\leq d(\Gamma_l)/2$
this operator is boundedly invertible and
\begin{equation}
\label{Wm1bound}
\|W_{_A}(z,\Gamma^\fs)\|^{-1}\leq \frac{1}{1-\frac{\cV_K(\Gamma_l)}{d(\Gamma^\fs)^2/4}}<\infty.
\end{equation}
\end{lemma}
Note that finiteness of the bound \eqref{Wm1bound} follows from the
assumption \eqref{OcVKd}. Having compared  \eqref{WAzG} with
\eqref{WADx}, one may view the factorization result \eqref{MAGfact}
as a direct analog of the factorization \eqref{roots}.

\section{{Complex deformation of the block
operator matrix and
solvability\newline of the deformed Riccati equation}}
\label{Sscal}

In this section by $L$ we will understand an extension of one the
two celebrated Friedrichs models in \cite{Fried}, namely
the one discussed in \cite[Section 6]{Fried}. We assume that
$L$ is a $2\times2$ block operator matrix of the form \eqref{L}
where, from the very beginning, the entry $D$ is given in the spectral
representation and, thus, it reads as the operator of multiplication by
the independent variable. That is,
\begin{equation}
\label{D}
{(Df_{_D})(\lambda)=\lambda f_{_D}(\lambda)}, \quad
f_{_D}\in\fH_{_D}=L^2(\Delta\to\fh)\quad
\bigl(\Delta=(\alpha,\beta)\subset\bbR\bigr),
\end{equation}
where $\fh$ is an auxiliary Hilbert space and $L^2(\Delta\to\fh)$ is
formed by functions $f_{_D}$ that map $\Delta$ to $\fh$ and are such
that the $\fh$-norm $\|f(\lambda)\|_\fh$ is Lebesgue measurable and
square-integrable over $\Delta$,
\begin{equation}
\|f\|_{\fH_{_D}}:=\left(\int_\alpha^\beta d\lambda\|f(\lambda)\|^2_\fh\right)^{1/2}<\infty.
\end{equation}
The inner product in $\fH_{_D}=L^2(\Delta\to\fh)$ is defined by
\begin{equation}
\lal f,g\ral_{\fH_{_D}}= \int_\alpha^\beta d\lambda \lal f(\lambda),g(\lambda) \ral_{\fh},\quad
f,g\in\fH_{_D},
\end{equation}
where $\lal\cdot,\cdot\ral_{\fh}$ stands for the inner product in $\fh$.

The above definition of $D$ means that all its spectrum consists of
the single branch of the absolutely continuous spectrum that
uniformly covers the interval $[\alpha,\beta]$. We make no
specification of the entry $A$ except for that it is self-adjoint
and its spectrum is embedded into the interior of $\sigma_{_D}$, i.e. the inclusion
\eqref{Feshbach} holds, $\sigma_{_A}\subset(\alpha,\beta)$.
Necessarily, the coupling operator $C:\fH_{_A}\to\fH_{_D}$ acts as follows:
\begin{equation}
\label{C}
{(Cf_{_A})(\lambda)=\fc(\lambda)f_{_A}, \quad f_{_A}\in\fH_{_A}},
\end{equation}
with some $\cB(\fH_{_A},\fh)$-valued (for a.e.
$\lambda\subset\Delta$) function $\fc(\lambda)$. Also we assume that
\begin{equation}
\label{B}
{Bf_{_D}=\int_\Delta d\lambda\,\, \fb(\lambda)f_{_D}(\lambda)},
\end{equation}
with a $\cB(\fh,\fH_{_A})$-valued function $\fb(\lambda)$, $\lambda\in(\alpha,\beta)$.

Surely, if $L$ is self-adjoint then $C=B^*$ and, necessarily,
$\fc(\lambda)=\fb(\lambda)^*$ for a.e. $\lambda\in\Delta$. Notice
that a self-adjoint block operator matrix $L$ of the form \eqref{L}
with the entries $D$ and $B$ given by \eqref{D} and \eqref{B},
respectively, and $C=B^*$, but with the entry $A$ only having point spectrum was
discussed in \cite[Section 8]{MM99}.

Main assumption of the present section is the following hypothesis.
\begin{hypothesis}
\label{Hypo1} Assume that the operator-valued functions
$\fb:\,\Delta\to\cB(\fh,\fH_{_A})$ and
$\fc:\,\Delta\to\cB(\fH_{_A},\fh)$ are real analytic on the interval
$\Delta=(\alpha,\beta)\subset\bbR$ and admit analytic continuation from
$\Delta$ onto a domain $\cD\subset\bbC$, $\cD\supset\Delta$ (and
$\cD\not\supset\{\alpha,\beta\}$). Let $\cD^-=\cD\cap\bbC^-$ and
$\cD^+=\cD\cap\bbC^+$. Also assume that the following bounds hold:
\begin{subequations}
    \label{bsing1}
\begin{align}
\tag{\ref{bsing1}\,\textit{a}}\label{bsing1:a}
\|\fb(\mu)\|_{\cB(\fh,\fH_{_A})} & \leq c |\mu-\alpha|^{-\nu_{\fb}}\quad\text{and}\quad
\|\fc(\mu)\|_{\cB(\fH_{_A},\fh)} \leq c |\mu-\alpha|^{-\nu_{\fc}}\\
\nonumber
&\quad
\text{for any $\mu\neq\alpha$ lying in some open complex neighborhood of $\alpha$ in $\cD$},\\
\tag{\ref{bsing1}\,\textit{b}}\label{bsing1:b}
\|\fb(\mu)\|_{\cB(\fh,\fH_{_A})} & \leq c |\mu-\beta|^{-\nu_{\fb}} \quad\text{and}\quad
\|\fc(\mu)\|_{\cB(\fH_{_A},\fh)} \leq c |\mu-\beta|^{-\nu_{\fc}}\\
\nonumber
&\quad \text{for any $\mu\neq\beta$ lying in some open complex neighborhood of $\beta$ in $\cD$},
\end{align}
\end{subequations}
where $c$, $\nu_{\fb}$, and  $\nu_{\fc}$ are some constants, $c>0$ and $0\leq\nu_{\fb}<1/2$, $0\leq\nu_{\fc}<1/2$.
\end{hypothesis}

Under Hypothesis \ref{Hypo1}, let us consider various piecewise
smooth Jordan contours $\Gamma$ with fixed real end points $\alpha$ and
$\beta$, $\alpha<\beta$, obtained by continuous deformation from the interval
$\Delta=(\alpha,\beta)$ and lying either  completely  in $\cD^-\cup\Delta$ or
completely in $\cD^+\cup\Delta$. With every such a contour $\Gamma$ we associate
the Hilbert space $\fH_{_{D,\Gamma}}:=L^2(\Gamma\to\fh)$ formed by functions
$f_{_{D,\Gamma}}:\Gamma\to\fh$ that are square-integrable with respect to the
Lebesgue measure $|d\lambda|$ on $\Gamma$, that is, the inner product in
$\fH_{_{D,\Gamma}}$ is defined by
\begin{equation}
\lal f,g\ral_{\fH_{_{D,\Gamma}}}=\int_\Gamma |d\lambda| \lal f(\lambda),g(\lambda) \ral_{\fh}.
\end{equation}
Surely, the Hilbert space $\fH_{_D}$ is a particular case of
$\fH_{_{D,\Gamma}}$ for $\Gamma=\Delta$.
Then one introduces the following family of
block operator matrices:
\begin{equation} \label{LG}
L_{_\Gamma}=\left(\begin{array}{cc}
A & B_{_\Gamma} \\
C_{_\Gamma} & D_{_\Gamma}
\end{array}\right),
\end{equation}
where $D_{_\Gamma}$ is the operator of multiplication by the
independent variable $\lambda\in\Gamma$ in $L^2(\Gamma\to\fh)$, i.e.
\begin{equation}
\label{DG}
{(D_{_\Gamma}f_{_{D,\Gamma}})(\lambda)=\lambda f_{_{D,\Gamma}}(\lambda)}, \quad
f_{_{D,\Gamma}}\in\fH_{_{D,\Gamma}};
\end{equation}
the entry {$C_{_\Gamma}:\fH_{_A}\to\fH_{D,\Gamma}$} is defined as
\begin{equation}
\label{CG}
(C_{_\Gamma} f_{_A})(\lambda)=\fc(\lambda)f_{_A}, \quad f_{_A}\in\fH_{_A}, \quad \lambda\in\Gamma,
\end{equation}
and the entry $B_\Gamma:\fH_{D,\Gamma}\to\fH_{_A}$ as
\begin{equation}
\label{BG}
B_{_\Gamma} f_{_{D,\Gamma}}=\int_\Gamma d\lambda\,\, \fb(\lambda)f_{_{D,\Gamma}}(\lambda),\quad
f_{_{D,\Gamma}}\in\fH_{_{D,\Gamma}}.
\end{equation}
\begin{remark}
\label{DGnormal} Since the entry $D_{_\Gamma}$, defined by
\eqref{DG}, is the operator of multiplication by the independent
variable, it is normal, i.e. it satisfies
$D_{_\Gamma}^*D_{_\Gamma}=D_{_\Gamma}D_{_\Gamma}^*$. The spectrum
$\sigma_{D_\Gamma}$ of $D_{_\Gamma}$ is absolutely continuous and
occupies the closure $\overline{\Gamma}=\Gamma\cup\{\alpha,\beta\}$
of the contour $\Gamma$. The spectral measure $\sE_{D_{_\Gamma}}$ of the operator
$D_{_\Gamma}$ is given by
\begin{equation}
\label{EDGchi}
\bigl(\sE_{D_{_\Gamma}}(\delta)f\bigr)(\lambda)=\chi_{\delta}(\lambda)f(\lambda),\quad
\lambda\in\Gamma,\quad f\in\fH_{_{D,\Gamma}}=L^2(\Gamma\to\fh),
\end{equation}
where $\delta$ is an arbitrary Borel subset of $\overline{\Gamma}$
and $\chi_{\delta}$ denotes the characteristic function of $\delta$:
$\chi_{\delta}(\lambda)=1$ if $\lambda\in\delta$ and
$\chi_{\delta}(\lambda)=0$ if $\lambda\in\overline{\Gamma}\setminus\delta$.
\end{remark}
\begin{remark}
Unlike the spectra of $A$ and (original) $D$, the spectra of
$A$ and $D_{_\Gamma}$ are disjoint,
\begin{equation}
\label{dGd}
d(\Gamma)=\dist(\sigma_{_A},\sigma_{_{D_\Gamma}})=\dist(\sigma_{_A},\Gamma)>0,
\end{equation}
whenever $\Gamma\cap\Delta=\emptyset$.
\end{remark}
We interpret $L_{_\Gamma}$  defined by \eqref{LG}--\eqref{BG} as the
result of the \emph{complex deformation} of the original
operator $L=L_{\Delta}$ corresponding to $\Gamma=\Delta$. Note
that only the main-diagonal entry $D$
is varied while the other main-diagonal entry $A$ remains unchanged.
Similarly, the operator Riccati equations
\begin{align}
\label{RicAG}
{XA-D_{_\Gamma}X+XB_{_\Gamma}X}&{=C_{_\Gamma}},
\quad X\in\cB(\fH_{_A},\fH_{_{D_\Gamma}}),\\
\label{RicDG} {YD_{_\Gamma}-AY+YC_{_\Gamma}Y}&{=B_{_\Gamma}}, \quad
Y\in\cB(\fH_{_{D_\Gamma}},\fH_{_A})
\end{align}
associated with the block operator matrix $L_{_\Gamma}$ are called
the \emph{complexly deformed} Riccati equations. The deformation
is viewed as the one with respect to the original Riccati equations
\eqref{RicAG}, \eqref{RicDG} for $\Gamma=\Delta$.

As in Section \ref{sFactor}, for $\Gamma=\Gamma^-\subset\cD^-$ or
$\Gamma=\Gamma^+\subset\cD^+$,  we again denote by $\Omega(\Gamma)$
the domain lying inside the closed curve formed by the interval
$\Delta$ and contour $\Gamma$. Surely
$\Omega(\Gamma^\pm)\subset\cD^\pm$. When talking on the operator
$L_{_\Gamma}$, by resonances one understands the part of the point
spectrum $\sigma_p(L_{_\Gamma})$ of $L_{_\Gamma}$ lying in
$\Omega(\Gamma)$. The next lemma shows that the resonances lying in
the intersection of the domains $\Omega(\Gamma^\fs)$ for several
various contours $\Gamma^\fs$ with the same $l$ are common for the
respective operators $L_{\Gamma^\fs}$. This property is nothing but
the analog of the independence of resonances on the scaling
parameter in the standard complex scaling approach (see, e.g.,
\cite[Section XIII.10]{RS-III} and references therein).

\begin{lemma}
\label{dspecG}
Let the assumptions of Hypothesis \ref{Hypo1} hold.
Assume that $\Gamma^\fs_1$ and $\Gamma^\fs_2$, $\fs=\pm1$,
are piecewise smooth Jordan contours with the end points $\alpha$
and $\beta$, obtained by continuous deformation from $\Delta$ and
lying completely  in $\cD^\fs\cup\Delta$. Let $L_{_{\Gamma_j^\fs}}$,
$j=1,2$, be operators defined for the respective contours
$\Gamma_j^\fs$  by \eqref{LG}--\eqref{BG}. Then
$z\in\Omega(\Gamma^\fs_1)\cap\Omega(\Gamma^\fs_2)$ belongs to the point
spectrum $\sigma_p(L_{_{\Gamma_1^\fs}})$ of $L_{_{\Gamma_1^\fs}}$ if and
only if $z\in\sigma_p(L_{_{\Gamma_2^\fs}})$. Furthermore, if
$z\in\sigma_p(L_{_{\Gamma_1^\fs}})\cap\bigl(\bbC\setminus
\overline{\Omega(\Gamma_1^\fs)}\bigr)$ then $\lambda\in\sigma_p(L)$,
where $L=L_{_\Delta}$ is the original non-deformed operator block matrix
\eqref{LG} with $\Gamma=\Delta$.
\end{lemma}
\begin{proof}
Suppose that
\begin{equation}
\label{zLzO}
z\in\sigma_p(L_{_{\Gamma_1}})\quad\text{and}\quad z\in\Omega(\Gamma_1)\cap\Omega(\Gamma_2),
\end{equation}
where
$\Gamma_1=\Gamma_1^\fs$ and $\Gamma_2=\Gamma_2^\fs$ for a certain value of $l$, $\fs=\pm1$.
Let $f\neq 0$ be an eigenvector of $L_{_{\Gamma_1}}$ belonging to
the eigenvalue $z$, $f=(f_{_A},f_{_D})$ with $f_{_A}\in\fH_{_A}$ and $f_{_D}\in\fH_{_{D_{\Gamma_1}}}$.
Equation $L_{_{\Gamma_1}}f=zf$ is equivalent to
\begin{align}
\label{eqbs1}
(A-z)f_{_A}+B_{_{\Gamma_1}}f_{_D}&=0,\\
\label{eqbs2}
C_{_{\Gamma_1}}f_{_A}+(D_{_{\Gamma_1}}-z)f_{_D}&=0.
\end{align}
Taking into account \eqref{DG} and \eqref{CG}, from \eqref{eqbs2} we
obtain for $\lambda\in\Gamma_1$ (and automatically $\lambda\neq z$)
\begin{equation}
\label{fD}
f_{_D}(\lambda)=\frac{1}{\lambda-z}\fc(\lambda)f_{_A}.
\end{equation}
Initially, the formula $\eqref{fD}$ only works for
$\lambda\in\Gamma_1$. But, except for the point $z$, this formula
may be used to make an extension of $f_{_D}$ through the whole
domain where the $\cB(\fH_{_A},\fh)$-valued function $\fc$ is
defined and analytic. Then, under Hypothesis \ref{Hypo1} (which is
assumed) the extended $\fh$-valued function \eqref{fD} is well
defined and analytic in $\lambda\in\cD=\cD^-\cup\cD^+\cup\Delta$
except for the point $z$. Moreover, the following equality holds
\begin{equation}
\label{fDp} \fc(\lambda)f_{_A}+(\lambda-z)f_{_D}(\lambda)=0\quad
\text{for any }\,\lambda\in{\cD\setminus\{z\}}.
\end{equation}
At the same time, in view of \eqref{fD} the term
$B_{_{\Gamma_1}}f_{_D}$ on the left-hand side of \eqref{eqbs2} reads
\begin{equation}
\label{fA}
B_{_{\Gamma_1}}f_{_D}=\int_{\Gamma_1} d\mu \frac{\fb(\mu)\fc(\mu)}{\mu-z}f_{_A}.
\end{equation}
Since the function under the integration sign on the right-hand side
of \eqref{fA} is holomorphic in $\mu\in\cD\setminus\{z\}$, the
contour $\Gamma_1$ may be replaced, with no change in the integral
value, by any other piecewise continuous Jordan contour
$\Gamma\subset\cD^\fs$ obtained from $\Gamma_1$ by continuous
deformation without crossing the point $z$. In particular, since by the assumption
$z\in\Omega(\Gamma_1)\cap\Omega(\Gamma_2)$, the contour $\Gamma_2$
may be chosen for such a purpose and then the equality \eqref{eqbs1}
arises with $\Gamma_1$ replaced by $\Gamma_2$. Furthermore,
restricting \eqref{fDp} to $\lambda\in\Gamma_2$ results in
equality \eqref{eqbs2} rewritten for $L_{_{\Gamma_2}}$. Thus, we
have showed that \eqref{zLzO} implies
$z\in\sigma_p(L_{_{\Gamma_2}})$. Interchanging the roles of
$\Gamma_1$ and $\Gamma_2$ proves the converse implication and, thus,
completes the proof of the first statement of the lemma.

The remaining statement is proven in the same way by the observation
that for $z\in\sigma_p(L_{_{\Gamma_1}})\cap\bigl(\bbC\setminus
\overline{\Omega(\Gamma_1)}\bigr)$ one can equivalently replace in
\eqref{fA} integration over $L_{_{\Gamma_1}}$ by integration over
$\Delta$. In its turn, the equality \eqref{fDp} is also reduced to
$\lambda\in\Delta$. Thus we conclude that, in this case,
$z\in\sigma_p(L_{_{\Gamma_1}})$ implies $z\in\sigma_p(L)$,
completing the whole proof.
\end{proof}

Now let us consider the Schur complement
\begin{align}
\label{MAEG}
M_{_{A,\Gamma}}(z):=& A-z-B_{_\Gamma}(D_{_\Gamma}-z)^{-1}C_{_\Gamma}\\
=&A-z-\int_{\sigma_{_{D_\Gamma}}} B_{_\Gamma}\sE_{_{D_\Gamma}}(d\mu)C_{_\Gamma}\frac{1}{\mu-z}, \quad
z\not\in\sigma_{_{D_\Gamma}}=\overline{\Gamma},
\end{align}
corresponding to the block operator matrix $L_{_\Gamma}$. It is
straightforward to see that in the case under consideration
\begin{equation}
\label{BEC}
B_{_\Gamma}\sE_{_{D_\Gamma}}(d\mu)C_{_\Gamma}=\fb(\mu)\fc(\mu)\,d\mu
\end{equation}
and, thus,
\begin{align}
\label{MAEG1}
M_{_{A,\Gamma}}(z):=A-z-\int_\Gamma d\mu \frac{\fb(\mu)\fc(\mu)}{\mu-z}, \quad
z\in\bbC\setminus\overline{\Gamma},
\end{align}
Furthermore, the corresponding operator-valued function $K$ defined on $\bbR$ by \eqref{Kdef},
for $\mu\in\overline{\Delta}$ reads as
\begin{equation}
\label{Kbc}
K(\mu)=\int_\alpha^\mu d\lambda\,\, \fb(\lambda)\fc(\lambda), \quad \alpha\leq\mu\leq\beta.
\end{equation}
Under the Hypothesis \ref{Hypo1} the function $K$ admits an explicit analytic continuation onto the
domain $\cD$ simply by the formula
\begin{equation}
\label{Kbcg}
K(\mu)=\int_{\gamma(\alpha,\mu)} d\lambda\,\, \fb(\lambda)\fc(\lambda), \quad \mu\in\cD,
\end{equation}
where $\gamma(\alpha,\mu)$ stands for arbitrary piecewise Jordan contour having the ends
$\alpha$, $\mu$ and lying, except for the end point $\alpha$, completely in $\cD$. Therefore,
the derivative $K'(\lambda)$ is nothing but
\begin{equation}
\label{Kder}
{K'(\lambda)=\fb(\lambda)\fc(\lambda)}, \quad \lambda\in\cD.
\end{equation}
This means that the Schur complement $M_{_{A,\Gamma}}(\cdot)$ corresponding to $L_{_\Gamma}$
merely simply coincides with the function $M_{_A}(\cdot,\Gamma)$ defined in \eqref{MAK},
\begin{equation}
\label{Mcoin}
M_{_{A,\Gamma}}(z)=M_{_A}(z,\Gamma),\quad z\in\bbC\setminus\overline{\Gamma}.
\end{equation}

\begin{remark}
It is worth noting that the analyticity of the operator-valued
function $K'$ does not imply the analyticity of $\fb$ and $\fc$
(and, thus, in general it does not suggests the opportunity to
perform the complex transformation of $L$ of the type we did in this
section). This is seen from the two following elementary examples.
\end{remark}
\begin{example}
\label{ex1}
Let $\alpha=-1$, $\beta=1$, and $\fH_{_A}=\fh=\bbC$. Suppose that $\fb(\lambda)=\lambda-i$ and
$\fc(\lambda)=\frac{1}{\lambda-i}$, $\lambda\in(-1,1)$. Clearly, the product
$\fb(\lambda)\fc(\lambda)\equiv 1$ admits analytic continuation from the interval $(-1,1)$
to the whole complex plane $\bbC$ while the function $\fc$ isn't.
\end{example}
\begin{example}
\label{ex2} Let $\alpha$, $\beta$, $\fH_{_A}$, and $\fh$ be as in
Example \ref{ex1}. Suppose that
$\fb(\lambda)=\fc(\lambda)=|\lambda|$, $\lambda\in(-1,1)$. Clearly,
the product $\fb(\lambda)\fc(\lambda)\equiv \lambda^2$ admits
analytic continuation from the interval $(-1,1)$ to the whole
complex plane $\bbC$ while none of the functions $\fb$ and $\fc$ is
real analytic at the point $\lambda=0$.
\end{example}

Now we notice that, due to \eqref{EDGchi}, the
$\sE_{_{D_\Gamma}}$-norms of the operators $B_{_\Gamma}$ and
$C_{_\Gamma}$ (see Definition \ref{DefEN}) read as
\begin{align}
\label{BGEN}
\|B_{_\Gamma}\|_{\sE_{_{D_\Gamma}}}=&
\left(\int_\Gamma |d\mu|\, \|\fb(\mu)\fb(\mu)^*\|_{\cB(\fh)}\right)^{1/2}=
\left(\int_\Gamma |d\mu|\, \|\fb(\mu)\|^2_{\cB(\fh,\fH_{_A})}\right)^{1/2},\\
\label{CGEN}
\|C_{_\Gamma}\|_{\sE_{_{D_\Gamma}}}=& \left(\int_\Gamma |d\mu|\,
\|\fc(\mu)^*\fc(\mu)\|_{\cB(\fH_{_A})}\right)^{1/2}= \left(\int_\Gamma
|d\mu|\, \|\fc(\mu)\|^2_{\cB(\fH_{_A},\fh)}\right)^{1/2}.
\end{align}
In the case under consideration, from \eqref{Kder} it follows that
the quantity $\cV_{_K}(\Gamma)$, the variation \eqref{cVK} of the
$\cB(\fH_{_A})$-valued function $K$ along $\Gamma$, is explicitly
written as
\begin{equation}
\label{Gadmis1} \cV_{_K}(\Gamma)=\int_\Gamma
|d\mu|\,\|\fb(\mu)\fc(\mu).
\end{equation}
Together with \eqref{BGEN} and \eqref{CGEN}, this yields the bound
\begin{equation}
\label{Gadmis2} \cV_{_K}(\Gamma)=\int_\Gamma
|d\mu|\,\|\fb(\mu)\fc(\mu)\|\leq
\|B_{_\Gamma}\|_{\sE_{_{D_\Gamma}}}\|C_{_\Gamma}\|_{\sE_{_{D_\Gamma}}}.
\end{equation}

It is convenient to combine Hypothesis \ref{Hypo1} with our further assumptions
in the form of one more hypothesis.

\begin{hypothesis}
\label{Hypo2}
Assume Hypothesis \ref{Hypo1}. Assume in addition that there exist piecewise smooth Jordan contours
$\Gamma^-$ and/or $\Gamma^+$ with the end points $\alpha$ and $\beta$ such that
$\Gamma^\fs\subset\cD^\fs\cup\Delta$\, and
\begin{equation}
\label{BCED2d}
\|B_{_{\Gamma}}\|_{\sE_{_{D_{\Gamma}}}}\|C_{_{\Gamma}}\|_{\sE_{_{D_\Gamma}}}<\frac{d(\Gamma)^2}{4}\quad
\text{with}\quad \Gamma=\Gamma^\fs\quad\text{for both}\quad \fs=\pm1.
\end{equation}
\end{hypothesis}
\noindent In the following, the curves $\Gamma^\fs$, $\fs=\pm 1$, referred to in
Hypothesis \ref{Hypo2}, are called \emph{admissible} contours.

Under Hypothesis \ref{Hypo2} both deformed Riccati equations
\eqref{RicAG} and \eqref{RicDG} have their respective bounded
solutions $X_{_\Gamma}$ and $Y_{_\Gamma}$. Also the integral equation \eqref{EqMain} has
the unique solution $Z^\fs_{_A}$, $\fs=\pm1$, referred to in Theorem \ref{ThSolv1} and
this solution is directly related to the operator $X_{_\Gamma}$ associated with $\Gamma\subset\cD^\fs$.

\begin{theorem}
\label{ThFin} Assume  Hypothesis \ref{Hypo2}. Let $\Gamma=\Gamma^\fs$, $\fs=\pm1$,
be a Jordan contour from this hypothesis and set $d=d(\Gamma)$. Then the deformed
operator Riccati equations \eqref{RicAG} and \eqref{RicDG} have the
respective solutions $X_{_\Gamma}\in\cB(\fH_{_A},\fH_{_{D_\Gamma}})$
and $Y_\Gamma\in\cB(\fH_{_{D_\Gamma}},\fH_{_A})$ with the following properties:
\begin{itemize}
\item for $B_{_\Gamma}\neq 0$ the operator
$X_{_\Gamma}$ is
a unique solution to \eqref{RicAG} in the ball
\begin{equation}
\label{QEst2G}
\left\{T\in\cB(\fH_{_A},\fH_{_{D_\Gamma}})\,\big|\,\,\,
\|T\|<\|B_{_\Gamma}\|_{_{\sE_{D_\Gamma}}}^{-1}
\left(d-\sqrt{\|B_{_\Gamma}\|_{_{\sE_{D_\Gamma}}}\,\|C_{_\Gamma}\|_{_{\sE_{D_\Gamma}}}}\right)\right\}.
\end{equation}
The solution $X_{_\Gamma}$ has a finite $\sE_{_{D_\Gamma}}$--norm satisfying
the bound
\begin{equation}
\label{QEst1G}
\|X_{_\Gamma}\|_{\sE_{_{D_\Gamma}}}\leq
 \frac{1}{\|B_{_\Gamma}\|_{_{E_{D_\Gamma}}}}\,
\left(\frac{d}{2}-\sqrt{\frac{d^2}{4}-\|B_{_\Gamma}\|_{_{\sE_{D_\Gamma}}}\,\|C_{_\Gamma}\|_{_{\sE_{D_\Gamma}}}}\right).
\end{equation}
\item for $C_{_\Gamma}\neq 0$ the operator
$Y_{_\Gamma}$ is
a unique solution to \eqref{RicDG} in the ball
\begin{equation}
\label{QEst2GY}
\left\{T\in\cB(\fH_{_{D_\Gamma}},\fH_{_A})\,\big|\,\,\,
\|T\|<\|C_{_\Gamma}\|_{_{\sE_{D_\Gamma}}}^{-1}
\left(d-\sqrt{\|B_{_\Gamma}\|_{_{\sE_{D_\Gamma}}}\,\|C_{_\Gamma}\|_{_{\sE_{D_\Gamma}}}}\right)\right\}.
\end{equation}
The solution $Y_{_\Gamma}$ has a finite $\sE_{_{D_\Gamma}}$--norm satisfying
the bound
\begin{equation}
\label{QEst1GY}
\|Y_{_\Gamma}\|_{\sE_{_{D_\Gamma}}}\leq
 \frac{1}{\|C_{_\Gamma}\|_{_{E_{D_\Gamma}}}}\,
\left(\frac{d}{2}-\sqrt{\frac{d^2}{4}-\|B_{_\Gamma}\|_{_{\sE_{D_\Gamma}}}\,
\|C_{_\Gamma}\|_{_{\sE_{D_\Gamma}}}}\right).
\end{equation}
\end{itemize}
Furthermore, the solution $Z_{_A}^\fs$ of \eqref{EqMain} referred to
in Theorem \ref{ThSolv1} reads as
\begin{equation}
\label{ZABX}
Z_{_A}^\fs=A+B_{_{\Gamma}}X_{_{\Gamma}},\quad \fs=\pm 1,
\end{equation}
and is independent of the Jordan contour
$\Gamma\subset\cD^\fs\cup\Delta$ satisfying the assumptions of Hypothesis
\ref{Hypo2}.
\end{theorem}
\begin{proof}
The hypothesis includes the bound \eqref{BCED2d} which by
\eqref{STEN} implies both \eqref{BCest} and \eqref{BCest1}. Then
the statements concerning the solutions $X_{_\Gamma}$ and $Y_{_\Gamma}$
to the Riccati equations \eqref{RicAG} and \eqref{RicDG} follow from
Theorem \ref{QsolvN} and \eqref{QsolvN1}, respectively.

In the case under consideration, the transformed Riccati equation \eqref{Xeq} for
$X_{_\Gamma}$  is as follows:
\begin{equation}
\label{XeqG}
X_{_\Gamma}=\int_\Gamma \sE_{_{D_\Gamma}}(d\mu)C_{_\Gamma}(A+B_{_\Gamma}X_{_\Gamma}-\mu)^{-1}.
\end{equation}
Notice that by \eqref{QEst1G} we have
\begin{equation}
\label{BXbound}
\|B_{_\Gamma}X_{_\Gamma}\|\leq\|B\|_{_{\sE_{D_\Gamma}}}\|X\|_{_{\sE_{D_\Gamma}}}<d(\Gamma)/2.
\end{equation}
Taking to account that $A$ is a self-adjoint operator, this implies
\begin{equation}
\label{resBXbound} \|(A+B_{_\Gamma}X_{_\Gamma}-\mu)^{-1}\|\leq
\frac{1}{\dist(\mu,\sigma_{_A})-\|B_{_\Gamma}X_{_\Gamma}\|}<
\frac{2}{d(\Gamma)} \quad\text{for any}\,\, \mu\in\Gamma,
\end{equation}
which means that the integral on the right-hand side of
\eqref{XeqG} is well defined.  From \eqref{XeqG} one concludes
that  $Z_{_{A,\Gamma}}:=A+B_{_\Gamma}X_{_\Gamma}$
satisfies the equation
\begin{equation}
\label{ZGeq} Z_{_{A,\Gamma}}=A-\int_\Gamma d\mu\,
B_{_\Gamma}\sE_{_{D_\Gamma}}(d\mu)C_{_\Gamma}(Z_{_{A,\Gamma}}-\mu)^{-1}.
\end{equation}
In view of \eqref{BEC}, \eqref{Kbcg}, and \eqref{Kder} this equation is merely
the equation \eqref{EqMain}. Applying Theorem \ref{ThSolv1} then yields
$Z_{_A}^\fs=Z_{_{A,\Gamma}}$ which is just \eqref{ZABX}. The independence of
$Z_{_A}^\fs$ on contours
$\Gamma\subset\cD^\fs\cup\Delta$ satisfying the assumptions of Hypothesis
\ref{Hypo2} is proven by Lemma \ref{HuniqueCor}.
\end{proof}

\begin{corollary}
\label{CorDiag} Under Hypothesis \ref{Hypo2} the partially deformed
block operator matrix $L_{_\Gamma}$ defined by
\eqref{LG}--\eqref{BG} for
$\Gamma=\Gamma^\fs\subset\cD^\fs\cup\Delta$, $\fs=\pm1$, admits the block
diagonalization \eqref{Ldiag}  in terms of the unique solutions
$X_{_\Gamma}$ and $Y_{_\Gamma}$ referred to in Theorem \ref{ThFin}.
In particular,
\begin{equation}
\label{LdiagG}
L_{_\Gamma}=(I+Q_{_\Gamma})\left(\begin{array}{cc}
Z_{_{A}}^\fs & 0 \\
0 & Z_{_{D,\Gamma}}
\end{array}\right)(I+Q_{_\Gamma})^{-1},
\end{equation}
where $Z_{_{A}}^\fs$ is the operator \eqref{ZABX}, $Q_{_\Gamma}$ is given by
\begin{equation}
\label{QZDG}
Q_{_\Gamma}=\left(\begin{array}{cc}
0 & Y_{_\Gamma} \\
X_{_\Gamma} & 0
\end{array}\right),\quad\text{and}\quad Z_{D,\Gamma}=D_{_\Gamma}+C_{_\Gamma}Y_{_\Gamma}.
\end{equation}
\end{corollary}
\begin{proof}
By the hypothesis, the condition \eqref{BCED2d} holds. Then the statement is
proven by applying first Corollary \ref{CorN1} and then Remark \ref{Rem1}.
\end{proof}

\begin{remark}
It is worth noting that, because of \eqref{EDGchi}, the solution
$X_{_{\Gamma^\fs}}$, $\fs=\pm1$, represents an operator from $\fH_{_A}$ to
$\fH_{_{D,\Gamma}}=L^2(\Gamma\to\fh)$ whose action is given by
\begin{equation}
\label{Xx} {(X_{_{\Gamma^\fs}}f_{_A})(\lambda)=\fx^\fs(\lambda)f_{_A},
\quad  f_{_A}\in \fH_{_A},\,\, \lambda\in\Gamma^\fs},\,\, \fs=\pm1,
\end{equation}
where the $\cB(\fH_A,\fh)$-valued function $\fx^\fs$ of the complex
variable $\lambda\in\cD^\fs\setminus\sigma_{Z^\fs_A}$ reads
\begin{equation}
\label{fxl}
\fx^\fs(\lambda):=\fc(\lambda)(Z^\fs_{_A}-\lambda)^{-1}, \quad \lambda\in\cD^\fs\setminus\sigma_{Z^\fs_A}.
\end{equation}
Similarly, the operator $Y_{_{\Gamma}}$, for $\Gamma=\Gamma^\fs$ with fixed $\fs=\pm1$, may be presented as
\begin{equation}
\label{Yyl}
Y_{_{\Gamma}}f_{_{D,\Gamma}}=\int_{\Gamma}d\mu\, \fy^\fs(\mu) f_{_{D,\Gamma}}(\mu),
\end{equation}
where $f_{_{D,\Gamma}}\in\fH_{_{D,\Gamma}}$ and $\fy^\fs:\Gamma\to\cB(\fh,\fH_A)$ is
the operator-valued function given by
\begin{equation}
\label{fyl}
\fy^\fs(\lambda)=-(\widetilde{Z}_A^\fs-\lambda)^{-1}\fb(\lambda),\quad
\lambda\in\cD^\fs\setminus\sigma_{\widetilde{Z}^\fs_A}\,\,
\bigl(=\cD^\fs\setminus\sigma_{Z^\fs_A}\bigr),
\end{equation}
with $\widetilde{Z}_A^\fs=A-Y_{_\Gamma}C_{_\Gamma}$.
By \eqref{ZZa} the operators $\widetilde{Z}_A^\fs$ and $Z_{_A}^\fs$ are similar to each other,
\begin{equation}
\label{ZZsim}
\widetilde{Z}_{_A}^\fs=(I-Y_{_\Gamma}X_{_\Gamma})Z_{_A}^\fs(I-Y_{_\Gamma}X_{_\Gamma})^{-1}.
\end{equation}
This means that, like $Z_{_A}^\fs$,
the operator $\widetilde{Z}_A^\fs$ does not depend on a Jordan contour
$\Gamma\subset\cD^\fs\cup\Delta$ satisfying the assumptions of
Hypothesis \ref{Hypo2}. The functions $\fx^\fs$ and $\fy^\fs$ work
simultaneously for all the piecewise Jordan contours $\Gamma^\fs$ with
$\Omega(\Gamma^\fs)\subset\cD^\fs$ such that
$\sigma_{Z^\fs_A}\subset\Omega(\Gamma^\fs)$.
\end{remark}

\begin{remark}
By \eqref{LdiagG} the spectrum of the operator $L_{_\Gamma}$ is nothing but the
union of the spectra of $Z_{_{A}}^\fs$ and $Z_{_{D,\Gamma}}$:
\begin{equation}
\label{2comp}
\sigma_{_{L_{_\Gamma}}}=\sigma_{_{Z_A^\fs}}\cup\sigma_{_{Z_{_{D,\Gamma}}}}, \qquad \Gamma=\Gamma^\fs, \,\, \fs=\pm1.
\end{equation}
Under the hypothesis of Theorem \ref{ThFin} we have the bound
\eqref{BXbound} for the product $B_{_{\Gamma}}X_{_{\Gamma}}$. Similarly, by \eqref{QEst1GY}
one obtains
\begin{equation}
\label{CYbound}
\|C_{_\Gamma}Y_{_\Gamma}\|\leq
\|C_{_\Gamma}\|_{_{\sE_{D_\Gamma}}}\|Y_{_\Gamma}\|_{_{\sE_{D_\Gamma}}}<d(\Gamma)/2.
\end{equation}
Since the operator $A$ is self-adjoint and $D_{_\Gamma}$ is normal, from
\eqref{BXbound} and \eqref{CYbound} it respectively follows that
$\sigma_{_{Z_A^\fs}}\subset\cO_{_{d(\Gamma)/2}}(\sigma_{_A})$ and
$\sigma_{_{Z_{_{D,\Gamma}}}}\subset\cO_{_{d(\Gamma)/2}}(\sigma_{_{D_\Gamma}})$.
Hence, the spectra $\sigma_{_{Z_A^\fs}}$ and $\sigma_{_{Z_{_{D,\Gamma}}}}$
are disjoint,
\begin{equation}
\dist(\sigma_{_{Z_A^\fs}},\sigma_{_{Z_{_{D,\Gamma}}}})>0.
\end{equation}
Recall that the set $\sigma_{_{Z_A^\fs}}$ depends on $l$ but does not
depend on (the form of) the contour $\Gamma^\fs$ satisfying \eqref{OcVKd} (see
Theorem \ref{ThFin}; cf. Lemma \ref{HuniqueCor}).
\end{remark}

\end{document}